\documentclass[11pt]{article}
\usepackage{graphicx, colordvi}
\usepackage{amsfonts}
\usepackage{amssymb}
\usepackage{amsthm,cite,color}
\usepackage{dsfont}
\usepackage{epsfig}
\usepackage{mathrsfs}
\usepackage{amsfonts}
\usepackage{amssymb}
\usepackage{amsmath}
\usepackage{amssymb,amsfonts,amsmath,amsthm,cite,color}
\usepackage{dsfont}
\usepackage{epsfig}
\usepackage{mathrsfs}
\usepackage{longtable}
\usepackage{hyperref}
\hypersetup{
colorlinks=true,
linkcolor=cyan,
filecolor=blue,
urlcolor=red,
citecolor=green,
}

\newtheorem{theo}{Theorem}
\newtheorem{rem}{Remark}
\newtheorem{prop}[theo]{Proposition}
\newtheorem{prob}[theo]{Problem}
\newtheorem{coro}[theo]{Corollary}
\newtheorem{defi}[theo]{Definition}
\newtheorem{lem}[theo]{Lemma}
\newtheorem{exm}[theo]{Example}

\makeatletter \@addtoreset{equation}{section}
\@addtoreset{theo}{section} \makeatother

\newcommand{\bN} { {\mathbb{N}}}

\newcommand{\bQ} { {\mathbb{Q}}}
\newcommand{\bZ} { {\mathbb{Z}}}

\newcommand{\bF} { {\mathbb{F}}}
\newcommand{\bK} { {\mathbb{K}}}

\DeclareMathOperator{\ann}{ann}
\DeclareMathOperator{\hdeg}{hdeg}
\DeclareMathOperator{\ldeg}{ldeg}

\newcommand{\qbinom}[2]{\left[ \begin{matrix} #1 \\ #2 \end{matrix} \right]}
\def\qed{\hfill \rule{4pt}{7pt}}
\def\pf{\noindent {\it Proof.} }

\begin{document}
\begin{center}

 {\large \bf Polynomial reduction for $q$-holonomic sequences}
\end{center}
\begin{center}
{  Rong-Hua Wang}$^{1}$, {Xiao-Ran Yang}$^{2}$  and {Michael X.X. Zhong}$^{3}$

   $^{1,2}$School of Mathematical Sciences\\
   Tiangong University \\
   Tianjin 300387, P.R. China \\
   $^1$wangronghua@tiangong.edu.cn \\
   $^2$2430141422@tiangong.edu.cn \\[10pt]

   $^3$School of Mathematical Sciences\\
   Tianjin University of Technology \\
   Tianjin 300384, P.R. China\\
   zhong.m@tjut.edu.cn
\end{center}

\vskip 6mm \noindent {\bf Abstract.}

This paper provides a (Laurent) polynomial reduction to $q$-holonomic sequences $F_k(q)$.
We first characterize Laurent polynomials $\tilde{p}(x)$ such that the product $\tilde{p}(q^k)F_k(q)$ is summable.
Then the reduction framework is given to decompose any given Laurent polynomial into a summable part and a remainder with lower degree.
Finally, we introduce a power-partible reduction for $q$-holonomic sequences of which the recurrence relation satisfies a certain symmetry condition.
The advantage is that it can not only simultaneously eliminate the highest-degree and lowest-degree terms of a Laurent polynomial satisfying a symmetry condition, but also guarantee the symmetry of the remainder.
As applications, we apply the reduction to $q$-central-Delannoy numbers to derive new $q$-identities and $q$-congruences.


\section{Introduction}

The mechanical proof and discovery of combinatorial identities have been central themes in symbolic computation, a paradigm largely established by the seminal work of Wilf and Zeilberger \cite{WZ1990,Zeilberger1990c,Zeil1990,Zeilberger1991}.
The method of creative telescoping, as the core of the WZ theory, provides a powerful algorithmic framework for proving hypergeometric identities.
Within this framework, reduction-based methods play a crucial role by separating the computation of telescopers from the potentially costly calculation of certificates \cite{CHEN2019}, thereby enabling more efficient algorithm design.
This reduction-based approach traces its roots back to the classical work of Ostrogradsky \cite{O1845} and Hermite \cite{Hermite1872} on the integration of rational functions, and has been systematically extended to various classes of functions, including multivariate rational functions \cite{BCCL2010, BLS2013}, hyperexponential functions \cite{BCCLX2013}, algebraic functions \cite{CKK2016}, and D-finite functions \cite{CHKK2018, BCLS2018}.

In the discrete setting, Abramov \cite{Abramov1975} first studied the summability problem for rational functions, which was extended to the trivariate rational case in \cite{CHHLW} with the key idea being to decompose a given rational function into a summable part and a non-summable remainder. Subsequently, Abramov and Petkov\v{s}ek \cite{AP2001, AP2002} and van der Hoeven \cite{Hoeven2018} extended this decomposition to hypergeometric terms and holonomic cases respectively. In recent years, significant advances have been made in reduction-based methods. Chen et al. \cite{CHKL2015} proposed a modified Abramov--Petkov\v{s}ek reduction, which significantly improved algorithmic efficiency by introducing the concept of polynomial reduction, and successfully applied it to the computation of minimal telescopers for bivariate hypergeometric terms.

Of particular interest is the polynomial reduction introduced by Hou, Mu, and Zeilberger \cite{HouMuZeil2021} in 2021, which avoids the multiplicative decomposition needed in \cite{CHKL2015}. This method not only enables the automated proof of hypergeometric identities but also facilitates the generation of infinite families of supercongruences. This polynomial reduction was utilized by Hou and Li \cite{HouLi2021} to obtain new hypergeometric identities.

In the realm of $q$-analogues, Wang and Zhong \cite{WZ2022} developed a $q$-analogue of this reduction, establishing a $q$-rational reduction framework via the $q$-Gosper representation that characterizes the summability of products of rational functions with $q$-hypergeometric terms. As applications, they derived several $q$-analogues of series for $\pi$ and provided a new tool for the automated discovery of $q$-identities.

More recently, Wang and Zhong \cite{WZ2023} extended the polynomial reduction framework to the broader class of holonomic sequences, which includes harmonic numbers, Fibonacci numbers, Domb numbers, Franel numbers, and all hypergeometric terms.
By introducing a difference Lagrange identity, they established a polynomial reduction for holonomic sequences and successfully applied it to derive new $\pi$-series involving Domb numbers and Franel numbers, as well as new families of congruences.

A refined property of polynomial reduction, known as power-partibility, was first observed by Hou, Mu, and Zeilberger \cite{HouMuZeil2021} in the hypergeometric setting, where the reduced polynomial contains only odd or only even powers of $(k-\gamma)$ under a symmetry condition. This property has been proven instrumental in deriving systematically families of supercongruences.
Subsequently, Wang and Zhong \cite{WZ2025, WZ2025Delannoy} extended power-partibility to holonomic sequences, establishing a comprehensive theory with applications to Ap\'ery numbers, Legendre polynomials, and related congruences.

The present paper further extends the polynomial reduction to the $q$-holonomic sequences, a broad class that includes $q$-hypergeometric terms and many other important combinatorial sequences.
We first establish a polynomial reduction framework for $q$-holonomic sequences via the adjoint operator and a corresponding Lagrange identity, characterizing the summability of terms of the form $\tilde{p}(q^k)F_k(q)$, where $\tilde{p}(x)$ is a Laurent polynomial.
As applications, we apply the developed theory to $q$-central-Delannoy numbers to derive new $q$-identities and $q$-congruences.
Finally, we introduce the notion of $q$-power-partibility for $q$-holonomic sequences. A $q$-holonomic sequence $F_k(q)$ is called $q$-power-partible if the coefficients of its associated linear recurrence satisfies a certain symmetry condition.
Under this condition, we develop a reduction algorithm for Laurent polynomials $Q([k])$ satisfying
\[
Q([\gamma+k])=\pm Q([\gamma-k]),
\]
for some rational number $\gamma$.
The key feature of this algorithm is, by appropriately choosing reduction functions, eliminating the highest and lowest degree terms simultaneously.
The procedure can be repeated until no further reduction is possible.
As an application, we generalize a theorem of Pan \cite{Pan2014} from the linear case to arbitrary odd powers, establishing new $q$-congruences.

The remainder of this paper is organized as follows. Section 2 introduces the fundamental concepts of $q$-holonomic sequences and establishes the theoretical foundation of polynomial reduction. Section 3 is devoted to $q$-power-partible reduction, with applications in deriving $q$-congruences on $q$-central Delannoy numbers.

\section{Polynomial reduction with applications}

In this section, we introduce the polynomial reduction for $q$-holonomic sequences and apply it to derive $q$-congruences on $q$-central-Delannoy numbers.

\subsection{Polynomial reduction for $q$-holonomic sequences}

Let $\bK$ be a field of characteristic zero and $q$ an inderterminate.
A sequence $(F_k(q))_{k=0}^{\infty}$ is called \emph{$q$-holonomic} if there exist $J\in\bN=\{0,1,2,\ldots,\}$ and Laurent polynomials $a_i(x)\in\bK(q)[x,x^{-1}]$, $0\leq i\leq J$, with $a_J(x)\neq 0$ such that
\begin{equation}\label{eq:rec homo}
\sum_{i=0}^{J}a_i(q^k)F_{k+i}(q)=0.
\end{equation}
Or, equivalently, if we define the \emph{annihilator} of $F_k(q)$ by
\[
\ann F_k(q):=\left\{L=\sum_{i=0}^{J}a_i(q^k)\sigma^i\in \bK(q)[q^k,q^{-k}][\sigma]\mid L(F_k(q))=0\right\},
\]
where $\sigma$ is the shift operator,
that is, $\sigma^{i} F_k(q)=F_{k+i}(q)$ for any $i\in\bZ$,
then $(F_k(q))_{k=0}^{\infty}$ is $q$-holonomic if and only if
$\ann F_k(q)\neq \{0\}$.
We call $J$ in \eqref{eq:rec homo} the order of the recurrence relation for $F_k(q)$, and the minimum order of all such recurrences is called the \emph{order} of $F_k(q)$.

The class of $q$-holonomic sequences covers a great percentage of combinatorial sequences arising in applications.
In particular, all $q$-hypergeometric terms are $q$-holonomic.
Recently, Wang and Zhong \cite{WZ2022} introduced a reduction for $q$-hypergeometric terms which is a $q$-analogue of the Hou--Mu--Zeilberger reduction introduced in 2021.
For a $q$-hypergeometric term $t_k(q)$, a key step in the reduction is to characterize such polynomials $\tilde{p}(x)\in\bK(q)[x]$ that the product $\tilde{p}(q^k)t_k(q)$ is \emph{$q$-Gosper-summable}, that is,
\[\tilde{p}(q^k)t_k(q)=\Delta (u(q^k)t_k(q)),\]
for some rational function $u(x)\in\bK(q)(x)$, where $\Delta=\sigma-1$ is the difference operator (that is, $\Delta F_k(q)=F_{k+1}(q)-F_k(q)$).

It is natural to consider a similar problem in the $q$-holonomic case.
\begin{prob}
Given a $q$-holonomic sequence $(F_k(q))_{k=0}^{\infty}$ satisfying \eqref{eq:rec homo},
for which Laurent polynomials $\tilde{p}(x)\in\bK(q)[x,x^{-1}]$, the product $\tilde{p}(q^k)F_k(q)$ can be written as
\begin{equation}\label{eq: Abramov--van-Hoeij}
\tilde{p}(q^k)F_k(q)=\Delta\left(\sum_{i=0}^{J-1}u_i(q^k)F_{k+i}(q)\right)
\end{equation}
for some rational functions $u_0(x),u_1(x),\ldots,u_{J-1}(x)\in\bK(q)(x)$?
\end{prob}
To proceed, some notations are needed.
Given $L=\sum_{i=0}^{J}a_i(q^k)\sigma^i$ with $a_i(x)\in\bK(q)[x,x^{-1}]$,
the \emph{adjoint} of $L$ is defined by
\begin{equation*}\label{eq:L}
L^{\ast}=\sum_{i=0}^{J}a_i(q^{k-i})\sigma^{-i}.
\end{equation*}
Then
$
L^{\ast}(p(q^k))=\sum_{i=0}^{J}a_i(q^{k-i})p(q^{k-i})
$
for any $p(x)\in\bK(q)(x)$.
For simplicity of illustration, we also utilize
\[L^{\ast}(p(x))=\sum_{i=0}^{J}a_i(q^{-i}x)p(q^{-i}x)\]
when there is no confusion.
The following lemma due to van der Hoeven \cite[Proposition 3.2]{Hoeven2018} is crucial in our polynomial reduction.

\begin{lem}\label{lem: summable}
Let $L=\sum_{i=0}^{J}a_i(q^k)\sigma^i$ with $a_i(x)\in\bK(q)[x,x^{-1}]$.
Then for any $p(x)\in\bK(q)(x)$ and sequence $F_k(q)$, we have
\begin{equation}\label{eq:Lagrange identity}
p(q^k)L(F_{k}(q))-L^{\ast}(p(q^k))F_{k}(q)=\Delta
\left(\sum_{i=0}^{J-1}u_i(q^k)F_{k+i}(q)\right),
\end{equation}
where
\begin{equation}\label{eq:u_i}
u_i(q^k)=\sum_{j=1}^{J-i}a_{i+j}(q^{k-j})p(q^{k-j}).
\end{equation}
\end{lem}

When $L\in\ann F_k(q)$, identity \eqref{eq:Lagrange identity} reduces to
\begin{equation}\label{eq: summable}
L^{\ast}(p(q^k))F_k(q)=\Delta
\left(-\sum_{i=0}^{J-1}u_i(q^k)F_{k+i}(q)\right).
\end{equation}
If $\tilde{p}(x)=L^{\ast}(p(x))\in\bK(q)[x,x^{-1}]$, then $\tilde{p}(x)$
is a desired polynomial such that \eqref{eq: Abramov--van-Hoeij} holds.
From equality \eqref{eq: summable} we obtain
\begin{equation}\label{eq:rec ann}
\sum_{k=0}^{n-1}L^{\ast}(p(q^k)) F_k(q)
=\left(\sum_{i=0}^{J-1}u_i(1)F_i(q)\right)
-\left(\sum_{i=0}^{J-1}u_i(q^n)F_{n+i}(q)\right),
\end{equation}
where $u_i$ is defined in \eqref{eq:u_i}.

When the order $J$ in \eqref{eq:rec homo} is minimum, it can be directly checked that all rational functions $\tilde{p}(x)\in\bK(q)(x)$ such that \eqref{eq: Abramov--van-Hoeij} holds are of the form $L^{\ast}(p(x))$ with $p(x)\in\bK(q)(x)$.
It may happen that $p(x)\in\bK(q)(x)\setminus\bK(q)[x,x^{-1}]$ but $L^{\ast}(p(x))\in\bK(q)[x,x^{-1}]$.
The following lemma shows that this rarely happens.
\begin{lem}\label{lem:rational->polynomial}
Let $L=\sum_{i=0}^{J}a_i(q^k)\sigma^i$ with $a_i(x)\in \bK(q)[x,x^{-1}],\ 0\leq i\leq J$ and $a_0(x)a_J(x)\neq 0$.
If $p(x)\in\bK(q)(x)$ and
\begin{equation}\label{eq:coprime}
	\gcd (a_0(x),a_J(q^i x))=1,\quad\forall i\in\bN,
\end{equation}
then $L^{\ast}(p(x))\in\bK(q)[x,x^{-1}]$ if and only if $p(x)\in\bK(q)[x,x^{-1}]$.
\end{lem}
\pf
Suppose $p(x)=\frac{r(x)}{s(x)}$ with $r(x),s(x)\in\bK(q)[x,x^{-1}]$ and $r(x),s(x)$ are coprime in $\bK(q)[x,x^{-1}]$.
If $\tilde{p}(x)=L^{\ast}(p(x))$ is a Laurent polynomial in $\bK(q)[x,x^{-1}]$, then
\begin{equation*}\label{eq:ispolynomial}
	\tilde{p}(x)
	=\sum_{i=0}^{J}a_i(q^{-i}x)p(q^{-i}x)
	=\sum_{i=0}^{J}a_i(q^{-i}x)\frac{r(q^{-i}x)}{s(q^{-i}x)}.
\end{equation*}
Multiplying $s(x)s(q^{-1}x)\cdots s(q^{-J}x)$ on both ends, we obtain
\begin{align*}
	\tilde{p}(x)\prod_{i=0}^{J}s(q^{-i}x)
	=\sum_{i=0}^{J}a_i(q^{-i}x)r(q^{-i}x)\prod_{\substack{0\leq j\leq J\\j\neq i}} s(q^{-j}x).
\end{align*}
Apparently, $s(x)$ is a divisor of the left hand side, then we must have
\begin{align*}
	s(x)\mid a_0(x)r(x)s(q^{-1}x)\cdots s(q^{-J}x).
\end{align*}
From $\gcd(r(x),s(x))=1$ we know that
\begin{equation}\label{eq: division1}
	s(x)\mid a_0(x)s(q^{-1}x)\cdots s(q^{-J}x).
\end{equation}
By a similar argument,
\[s(q^{-J}x)\mid a_J(q^{-J}x)s(x)s(q^{-1}x)\cdots s(q^{-J+1}x),\]
namely,
\begin{equation}\label{eq: division2}
	s(x)\mid a_J(x)s(q^J x)s(q^{J-1}x)\cdots s(qx).
\end{equation}

Suppose that $s(x)$ is not invertible in $\bK(q)[x,x^{-1}]$ and $t(x)$ is an irreducible factor of $s(x)$.
By \eqref{eq: division1}, if $t(x)\nmid a_0(x)$, we have $t(q^{j_1}x)\mid s(x)$ for some $j_1>0$.
Again, by \eqref{eq: division1}, if $t(q^{j_1}x)\nmid a_0(x)$, then $t(q^{j_1+j_2}x)\mid s(x)$ for some $j_2>0$.
Since $s(x)$ can not have infinitely many distinct irreducible factors, there must exist a $j\ge 0$ such that $t(q^jx)\mid a_0(x)$.
By a similar argument, \eqref{eq: division2} guarantees that there exists an $i\ge 0$ such that $t(q^{-i}x)\mid a_J(x)$.
Then $a_0(x)$ and $a_J(q^{i+j}x)$ have a common factor $t(q^{j}x)$, contradicting \eqref{eq:coprime}.
As the converse is clearly true, this completes the proof.
\qed

Combining Lemma \ref{lem: summable} and Lemma \ref{lem:rational->polynomial}, we have the following result.
\begin{theo}\label{thm:rational to polynomial}
Let $F_k(q)$ be a $q$-holonomic sequence of order $J$ and $L=\sum_{i=0}^{J}a_i(q^k)\sigma^i\in\ann F_k(q)$ with
\begin{equation*}
	\gcd (a_0(x),a_J(q^i x))=1,\quad\forall i\in\bN.
\end{equation*}
Then $\tilde{p}(x)\in\bK(q)[x,x^{-1}]$ satisfies the summability condition \eqref{eq: Abramov--van-Hoeij} if and only if
$\tilde{p}(x)=L^{\ast}(p(x))$ for some $p(x)\in\bK(q)[x,x^{-1}]$.
\end{theo}

Next we will determine the degrees of $L^{\ast}(p(x))$ for $p(x)\in\bK(q)[x,x^{-1}]$ and $L=\sum\limits_{i=0}^{J}a_i(q^k)\sigma^i$ with coefficients \(a_i(x) \in \mathbb{K}(q)[x, x^{-1}]\).

For any Laurent polynomial $a(x)\in\bK(q)[x,x^{-1}]$,
denote by $\hdeg a(x)$ (resp. $\operatorname{ldeg} a(x)$) the highest (resp. lowest) degree of $x$ in $a(x)$ and
\[\deg a(x)=\hdeg a(x)-\ldeg a(x).\]
The \emph{highest degree} and the \emph{lowest degree} of $L$, denoted by $\hdeg L$ and $\ldeg L$ respectively, are defined as
\begin{equation*}\label{eq:d}
	\hdeg L =\max_{0\leq i \leq J} \{\hdeg a_i(x)\}
	\quad \text{and} \quad
	\operatorname{ldeg} L=\min_{0\leq i \leq J} \{\operatorname{ldeg} a_i(x)\}.
\end{equation*}
The \emph{degree} of $L$ is defined by
\begin{equation*}\label{eq:degL}
\deg L=\hdeg L-\ldeg L.
\end{equation*}
Let
\begin{equation*}\label{eq:f_L^{()t}}
f_L^{(t)}(s)=\sum_{i=0}^{J}[x^{t}](a_i(q^{-i}x))q^{-si},
\end{equation*}
where $[x^{t}](a(x))$ denotes the coefficient of $x^{t}$ in $a(x)$.
Define
\begin{equation*}\label{eq:nonnegative roots-h}
	R_L^{t}=\{s\in\bZ \mid f_L^{(t)}(s)=0\}.
\end{equation*}
Then $L$ is called \emph{highest-degree-degenerated} if $R_{L}^{\hdeg L}\neq\emptyset$ or \emph{lowest-degree-degenerated} if $R_{L}^{\ldeg L}\neq\emptyset$,
which will be denoted as \emph{HDD} or \emph{LDD} for short.
When $R_{L}^{\hdeg L}=R_{L}^{\ldeg L}=\emptyset$, we say $L$ is \emph{nondegenerated}.

Now we are ready to characterize the degrees of $L^{*}(p(x))$ for Laurent polynomials $p(x)$, which is crucial in the polynomial reduction.
\begin{theo}\label{thm:degree}
Let $L=\sum_{i=0}^{J}a_i(q^k)\sigma^i$  with $a_i(x)\in\bK(q)[x,x^{-1}]$ and $p(x)\in\bK(q)[x,x^{-1}]\setminus \{0\}$.
Then
\[
\hdeg L^{\ast}(p(x))=\hdeg L+\hdeg p(x)
\]
except when $L$ is HDD and $\hdeg p(x)\in R_L^{\hdeg L}$, in which case we have $\hdeg L^{\ast}(p(x))<\hdeg L+\hdeg p(x)$,
and
\[\ldeg L^{\ast}(p(x))=\ldeg L+\ldeg p(x)\]
except when $L$ is LDD and $\ldeg p(x)\in R_L^{\ldeg L}$, in which case we have $\ldeg L^{\ast}(p(x))>\ldeg L+\ldeg p(x)$.
\end{theo}
\pf
We know by definition that
$
	L^{*}(p(x)) = \sum_{i=0}^{J} a_i(q^{-i} x) p(q^{-i} x)
$
and
\[
\hdeg( a_i(q^{-i} x) p(q^{-i} x) ) = \hdeg  a_i(x) + \hdeg p(x) \leq \hdeg L + \hdeg p(x).
\]
Hence it is clear that
\[
	\hdeg L^{*}(p(x)) \leq \hdeg L + \hdeg p(x).
\]
Note that $\hdeg\ L^{\ast}(p(x)) < \hdeg L+\hdeg p(x)$ if and only if
\[
	\sum_{i=0}^{J} [x^{\hdeg L}](a_i(q^{-i}x)) [x^{\hdeg p(x)}](p(q^{-i}x))=0,
\]
which is equivalent to
\[
\sum_{i=0}^{J} [x^{\hdeg L}](a_i(q^{-i}x)) q^{-i\cdot\hdeg p(x)}
= f_L^{(\hdeg L)}(\hdeg p(x))=0.
\]
This completes the proof of the first part.	
The second part can be confirmed similarly and thus omitted.
\qed

Let $F_k(q)$ be a $q$-holonomic sequence and $L=\sum_{i=0}^{J}a_i(q^k)\sigma^i\in\ann F_k(q)$ with $a_0(q^k)a_J(q^k)\neq 0$.
By definition, one can see that $L$ is usually nondegenerated.
In this case, we can split any Laurent polynomial $Q(x)\in\bK(q)[x,x^{-1}]$ as
\[Q(x)=\tilde{p}(x)+\tilde{Q}(x),\]
where $\tilde{p}(x)$ is a Laurent polynomial satisfying the summability condition \eqref{eq: Abramov--van-Hoeij} and
$\deg\tilde{Q}(x)<\deg L$.
We call $\tilde{p}(x)$ the summable part and $\tilde{Q}(x)$ the reminder.
This clearly holds when $\deg Q(x)< \deg L$.
In the following, we assume that $\deg Q(x)\ge \deg L$.

First suppose that $L$ is not HDD.
We can choose a $p_0(x)\in\bK(q)[x,x^{-1}]$ with
$\hdeg p_0(x)=\hdeg Q(x)-\hdeg L$
and
$\ldeg p_0(x)\ge \ldeg Q(x)-\ldeg L$, for example, $p_0(x)=x^{\hdeg Q(x)-\hdeg L}$.
Then by Theorem \ref{thm:degree}, we have
\[
\hdeg L^*(p_0(x))=\hdeg Q(x)
\quad\text{and}\quad
\ldeg L^*(p_0(x))\ge \ldeg Q(x).
\]
It is clear that there exists a $b_0\in\bK(q)$ such that
$Q_1(x)=Q(x)-b_0L^*(p_0(x))$ is a Laurent polynomial satisfying
\begin{equation}\label{eq:reduction-h}
\hdeg Q_1(x)<\hdeg Q(x)
\quad\text{and}\quad
\ldeg Q_1(x)\ge \ldeg Q(x).
\end{equation}
Thus $\deg Q_1(x)<\deg Q(x)$.
If $\deg Q_1(x)<\deg L$, one can see that $Q(x)=L^*(b_0p_0(x))+Q_1(x)$ is the required decompositoin.
Otherwise, we can replace $Q(x)$ by $Q_1(x)$, and continue the above process.
After finite steps, one can find $b_s\in\bK(q)$ and Laurent polynomials $p_s(x),\tilde{Q}_h(x)\in\bK(q)[x,x^{-1}]$ such that
\begin{equation}\label{eq:reduction step-h}
Q(x)=\sum_{s}b_{s}L^*(p_{s}(x))+\tilde{Q}_h(x)
\end{equation}
with $\deg \tilde{Q}_h(x)<\deg L$, $\hdeg \tilde{Q}_h(x)< \hdeg Q(x)$ and $\ldeg \tilde{Q}_h(x)\ge \ldeg Q(x)$.
The equality \eqref{eq:reduction step-h} is the required decomposition since 
\[
\tilde{p}(x)=\sum_{s}b_{s}L^*\left(p_{s}(x)\right)
=L^*\left(\sum_{s}b_{s}p_{s}(x)\right)
\]
 satisfies \eqref{eq: Abramov--van-Hoeij} by Lemma \ref{lem: summable}.

When $L$ is not LDD,
we can choose a polynomial $p_0(x)\in\bK(q)[x,x^{-1}]$ with
$\ldeg p_0(x)=\ldeg Q(x)-\ldeg L$
and
$\hdeg p_0(x)\le \hdeg Q(x)-\hdeg L$,
for example, $p_0(x)=x^{\ldeg Q(x)-\ldeg L}$.
Then by Theorem \ref{thm:degree}, we have
\[
\ldeg L^*(p_0(x))=\ldeg Q(x)
\text{ and }
\hdeg L^*(p_0(x))\le \hdeg Q(x).
\]
There exists a $b_0\in\bK(q)$ such that $Q_1(x)=Q(x)-b_0L^*(p_0(x))$ is a Laurent polynomial satisfying
\begin{equation}\label{eq:reduction-l}
\ldeg Q_1(x)>\ldeg Q(x)
\text{ and }
\hdeg Q_1(x)\le \hdeg Q(x).
\end{equation}
Thus $\deg Q_1(x)< \deg Q(x)$.
A similar reduction to \eqref{eq:reduction step-h} then follows as
\begin{equation}\label{eq:reduction step-l}
Q(x)=\sum_{s}b_{s}L^*p_{s}(x)+\tilde{Q}_{\ell}(x)
\end{equation}
with $\deg \tilde{Q}_{\ell}(x)<\deg L$, $\ldeg \tilde{Q}_{\ell}(x)> \ldeg Q(x)$ and $\hdeg \tilde{Q}_{\ell}(x)\le \hdeg Q(x)$.

We call \eqref{eq:reduction step-h} or \eqref{eq:reduction step-l} a \emph{(Laurent) polynomial reduction} with respect to $L$.
It is worth to notice that, when $\deg Q(x)>\deg L$ and $L$ is nongenerated, that is, $L$ is neither HDD nor LDD,
we can choose $p_0(x)$ and $b_0$ such that both \eqref{eq:reduction-h} and \eqref{eq:reduction-l} hold.
This accelerates the reduction process.
In Section 3, we will present a natural scene for the application of this acceleration.

\subsection{Applications to $q$-central-Delannoy numbers}

In this subsection, we take the $q$-central-Delannoy number as an example to illustrate how the polynomial reduction can be utilized to derive arithmetic properties of combinatorial sequences.

For any $n\in\bZ$, the $q$-integer $[n]$ is defined by
\[
[n] = \frac{1 - q^n}{1 - q}.
\]
Clearly, $\displaystyle\lim_{q \to 1} [n] = n$.
For a nonnegative integer $n$ and a positive integer $k$, the $q$-binomial coefficient
\[
\qbinom{n}{k} = \frac{\prod_{1 \leq j \leq k} [n - j + 1]}{\prod_{1 \leq j \leq k} [j]}.
\]
Moreover, we set $\qbinom{n}{0}=1$ and  $\qbinom{n}{k}=0$ if $k<0$.

For $ m, n \in \mathbb{N}$, the Delannoy numbers
\[
D_{m,n} = \sum_{k \in \mathbb{N}} \binom{m}{k} \binom{n+k}{m}
\]
count lattice paths from $(0, 0)$ to $(m, n)$
in which only east $(1, 0)$, north $(0, 1)$, and northeast $(1, 1)$ steps are allowed.
The $n$th \emph{central Delannoy number}
$D_{n,n}$, denoted by $D_n$ for short, is then given by
\[
D_n=\sum_{k=0}^n \binom{n}{k} \binom{n+k}{k}
=\sum_{k=0}^n \binom{n}{k} \binom{2n-k}{n}.
\]
The arithmetic properties of central Delannoy numbers have been extensively studied.
Let $\varepsilon=\pm 1$ and $n$ a positive integer.
Sun \cite{Sun2012JNT} proved that
\begin{equation*}
\sum_{k=0}^{n-1}(2k+1)\varepsilon^k D_k \equiv 0 \pmod{n}.
\end{equation*}
When $p>3$ is a prime, Sun further derived that
\[
\sum_{k=0}^{p-1}(2k+1)(-1)^k D_k \equiv p-\frac{7}{12}p^4B_{p-3}\pmod {p^5}
\]
and
\[
\sum_{k=0}^{p-1}(2k+1)D_k \equiv p+2p^2q_p(2)-p^3q_p(2)^2\pmod {p^4},
\]
where $B_p$ are the Bernoulli numbers and $q_p(2)$ denotes the Fermat quotient $(2^{p-1}-1)/p$.
Later Guo and Zeng \cite{GuoZeng2012} showed that
\begin{equation}\label{eq:GuoZeng2v+1}
	\sum_{k=0}^{n-1}(2k+1)^{2v+1}\varepsilon^{k}D_k \equiv 0 \pmod{n}
\end{equation}
for any $v\in\bN$.

In the following, we will recall $q$-central-Delannoy numbers and apply the polynomial reduction to obtain their arithmetic properties.
In particular, we will obtain $q$-congruences which reduce to \eqref{eq:GuoZeng2v+1} by taking $q\to1$.

In 2014, Pan \cite{Pan2014} defined the generalized $q$-Ap\'ery polynomial
\begin{equation*}\label{Pan-Generlized}
	A_n^{(\alpha)}(x;q) = \sum_{k=0}^n q^{\alpha\left(\binom{k}{2} - kn\right)}
	\qbinom{n}{k}^{\alpha}
	\qbinom{n+k}{k}^{\alpha}x^k
\end{equation*}
and then proved
\begin{equation}\label{eq:Pan-A_k}
\sum_{k=0}^{n-1}q^{n-1-k}[2k+1]A_k^{(\alpha)}(x;q)^m\equiv 0\pmod {[n]},
\end{equation}
which confirmed a conjecture of Guo and Zeng \cite{GuoZeng2012a} by taking $q=1$.
One can see
\begin{equation}\label{q-Delannoy}
D_{n}(q):=A_n^{(1)}(1;q)= \sum_{k=0}^{n} q^{\binom{k}{2} - kn}
\qbinom{n}{k}\qbinom{n+k}{k}
\end{equation}
is a $q$-analogue of the $n$th central Delannoy number.

Let $\varepsilon\in\{ 1,-1 \}$, by $q$-Zeilberger algorithm, we find that
\begin{align}\label{qDe1L}
	L =  &q^{-k}(1+q^{k+1})[k+2] \sigma^2+ \varepsilon(q+q^{k+2} + q^{2k+4} + 3q^{k+3})[-2k-3] \sigma \nonumber\\
  &+  q^{-k} (1+q^{k+2})[k+1] \in \ann \varepsilon^kD_{k}(q).
\end{align}

\begin{theo}\label{eq:Theo-qDe1}
Let $L$ be given by \eqref{qDe1L}
and $n$ a positive integer.
For any Laurent polynomial $p(x) \in \bK(q)[x,x^{-1}]$ and $n\ge 1$, we have
\begin{align}\label{qDe1theo}
&    \sum_{k=0}^{n-1}  L^{\ast}(p([k])) \varepsilon^k D_{k}(q)\nonumber \\
=& \varepsilon^{n}[n]q^{1-n} \left((1 + q^{n+1}) p([n-1])\varepsilon D_{n-1}(q)-q(1 + q^{n-1})  p([n-2])D_{n}(q) \right).
\end{align}
\end{theo}
\pf
By equality \eqref{eq:rec ann} and the easily checked fact that
$u_{0}(1)D_{0}(q) +  u_{1}(1)\varepsilon D_{1}(q) = 0$,
we have
\begin{equation}\label{qD1fact}
	\sum_{k=0}^{n-1} L^{\ast}(p([k])) \varepsilon^{k}D_{k}(q) = -\left( u_{0}(q^n)\varepsilon^{n} D_{n}(q) + u_{1}(q^n) \varepsilon^{n+1}D_{n+1}(q)\right),
\end{equation}
where
$u_0(q^n) =  q^{2-n}(1 + q^{n-1})[n]p([n-2])-\varepsilon q^{-2n}(1 + q^n + 3q^{n+1} + q^{2n+1})[2n+1] p([n-1])  $ and
$u_1(q^n) = q^{1-n}(1 + q^n) [n+1]p([n-1]).$
As $L \in \ann \varepsilon^{k} D_{k}(q)$, one has
\begin{align}\label{bianliangdaihuan1}
&-q^{1-n}(1+q^n)[n+1]\varepsilon^{n+1}D_{n+1}(q)\nonumber\\
=&(q^2+q^{1-n})[n] \varepsilon^{n-1}D_{n-1}(q)+q(1+q^{n}+q^{2n+1}+3q^{n+1})[-2n-1] \varepsilon^{n+1}D_{n}(q)
\end{align}
for any $n \geq 1$.
Substituting \eqref{bianliangdaihuan1} into \eqref{qD1fact} derives \eqref{qDe1theo}.
\qed

\begin{coro}\label{LM:q-Delannoy}
		Let $L$ be given as in \eqref{qDe1L}.
For any $p(x) \in \mathbb{Z}[q,q^{-1}][x,x^{-1}]$, we have
	\[	
	\sum_{k=0}^{n-1} L^{\ast}(p([k])) \varepsilon^{k} D_{k}(q) \equiv 0 \pmod {[n]}.
	\]
\end{coro}
Taking different $p(x)$ in Corollary \ref{LM:q-Delannoy} leads to different $q$-congruences.
This provides a mechanical way to generate new $q$-congruences.
In the following, we show that even when $p(x)\in\mathbb{Q}[q,q^{-1}][x]$, Theorem \ref{eq:Theo-qDe1} can still give $q$-identities which are $q$-analogues of known congruences.

For $\varepsilon\in \{1,-1\}$, let
\[
\tilde{L} = (k+2)\sigma^2 -3\varepsilon(2k+3)\sigma + (k+1).
\]
It can be checked that $\tilde{L}(\varepsilon^k D_k)=0$.
The adjoint operator $\tilde{L}^{\ast}$ is given by
\[
\tilde{L}^{\ast} = k\sigma^{-2} -3\varepsilon(2k+1)\sigma^{-1} + (k+1).
\]
In \cite{WZ2025Delannoy}, the first and the third authors proved that 	\begin{equation}\label{Dkpt}
	\tilde{L}^{\ast}(x_s(k)) = (1-3\varepsilon)(2k + 1)^{s+1}
	+ 2 \sum_{j=1}^{\left\lfloor \frac{s+1}{2} \right\rfloor} e_j^{(s)} (2k + 1)^{s+1-2j},
	\end{equation}
where $ x_{s}(k) =(2k+3)^{s}$ and
$
	e_j^{(s)} = \binom{s}{2j} 2^{2j-1} + \binom{s}{2j-1} 2^{2j-2} \in \mathbb{Z}
$
	for $ s \in \mathbb{N} $ and $ j = 1, 2, \dotsc, \left\lfloor \frac{s+1}{2} \right\rfloor $.
Taking $s=2v$ in \eqref{Dkpt} shows
\begin{equation*}\label{eq:(2k+1)^{2v+1}}
(2k + 1)^{2v+1} = \frac{1}{1-3\varepsilon} \tilde{L}^{\ast}(x_{2v}(k))
+ \frac{2}{3\varepsilon-1}\sum_{j=1}^{v} e_j^{(2v)}(2k+1)^{2(v-j)+1}.
\end{equation*}
Then by induction one can see for each $v\in\bN$,
\[
(2k+1)^{2v+1}=\tilde{L}^{\ast}\left(p^{(\varepsilon)}_v(k)\right)
\]
for a polynomial $p^{(\varepsilon)}_v(k)$ in $k$ satisfying
\[
p^{(\varepsilon)}_v(k) = \frac{1}{1-3\varepsilon} x_{2v}(k)
+ \frac{2}{3\varepsilon-1}\sum_{j=1}^{v} e_j^{(2v)}p^{(\varepsilon)}_{v-j}(k).
\]
Clearly $p^{(\varepsilon)}_0(k)=\frac{1}{1-3\varepsilon}$.
When $v\in\bN^{+}$, $p^{(\varepsilon)}_v(k)$ can be determined recursively.
By definitions of $x_{s}(k)$ and $e_j^{(s)}$, we know
\begin{equation*}\label{eq:p_v}
p^{(1)}_v(k)=2\sum_{j=1}^{2v} c^{(v)}_j k^{j} + \frac{c^{(v)}_0}{2}
\quad \text{and} \quad
p^{(-1)}_v(k)=\sum_{j=1}^{2v} t^{(v)}_j k^{j} + \frac{t^{(v)}_0}{4},
\end{equation*}
where $c^{(v)}_j,\ t^{(v)}_j \in \bZ$ for $0\leq j \leq 2v$.

\begin{exm}\label{ex:q-analogDelannoy}
Equality \eqref{qDe1theo} reduces to Guo and Zeng's congruence \eqref{eq:GuoZeng2v+1} when $p([k])=\frac{1}{2}  p^{(\varepsilon)}_v([k])$ and $q\to 1$.
\end{exm}
\pf
Taking $p([k])=\frac{1}{2}p^{(\varepsilon)}_v([k])$ and $q\to 1$ in \eqref{qDe1theo},
the left hand side becomes
\[
\sum_{k=0}^{n-1}\tilde{L}^{\ast}(p^{(\varepsilon)}_v(k))\varepsilon^{k}D_{k}
=\sum_{k=0}^{n-1}(2k+1)^{2v+1}\varepsilon^{k}D_{k},
\]
and the right hand side becomes
\[
\varepsilon^{n}n\left(\varepsilon p^{(\varepsilon)}_v(n-1)D_{n-1}-p^{(\varepsilon)}_v(n-2)D_{n}\right).
\]
Then by the definition of $p^{(\varepsilon)}_v(k)$, we get
\[
\sum_{k=0}^{n-1} (2k+1)^{2v+1} D_k
    \equiv \frac{c^{(v)}_0}{2} n \bigl( D_{n-1} - D_n \bigr) \equiv 0\pmod{n}
\]
and
\[
\sum_{k=0}^{n-1} (2k+1)^{2v+1} (-1)^kD_k
\equiv \frac{t^{(v)}_0}{4} (-1)^{n+1} n \bigl( D_{n-1} + D_n \bigr)\equiv 0 \pmod{n}
\]
since $2\mid (D_n - D_{n-1})$ and $4\mid (D_n +D_{n-1})$.
\qed

\section{$q$-Power-partibility and applications}

Power-partible reduction was initiated by Hou, Mu and Zeilberger \cite{HouMuZeil2021} in 2021 for hypergeometric terms.
This was later extended by Wang and Zhong~\cite{WZ2025,WZ2025Delannoy} to holonomic sequences and applied in mechanical discovery and proof of arithmetic properties of combinatorial sequences like Ap\'ery numbers and Delannoy numbers.

In this section, we introduce a $q$-power-partible reduction for $q$-holonomic sequences with annihilators satisfying a symmetry property.
For these $q$-holonomic sequences, we can not only apply the reduction introduced in Section 2 to Laurent polynomials to eliminate simultaneouly the highest and lowest terms, but also characterize the structure of the Laurent polynomial after reduction.

\subsection{$q$-Power-partible reduction for $q$-holonomic sequences}

In the following, we assume that $\bF=\bK(q^{\frac{1}{n}})$ for some $n\in\bN^+$.
Define $[\gamma]=\frac{1-q^\gamma}{1-q}$ if $q^\gamma\in \bF$.

\begin{defi}\label{def:q-power-partible}
Let	$L = \sum_{i=0}^{J} a_i([k])\sigma^i\in\bF[q^k,q^{-k}][\sigma]\in \ann F_k(q)$ with $a_0([k]) a_J([k]) \neq 0$.
If there exist an $\varepsilon \in \{1,-1\}$ and a $q^\gamma \in \bF$ such that
\begin{equation}\label{eq:gamma-symmetry}
a_i([\gamma + k]) = \varepsilon \cdot a_{J-i}([\gamma - k - J]),
\quad \forall k\in\bZ,
i = 0,1,\dots,J.
\end{equation}
Then $L$ or $F_k(q)$ is called \emph{$q$-power-partible with respect to $(\gamma,\varepsilon)$}.
\end{defi}

\begin{rem}\label{lem:frac-eps}
One only need to check equality \eqref{eq:gamma-symmetry} for $i=0,1,\dots,\left\lfloor \frac{J}{2} \right\rfloor$.
Since for
$\left\lfloor \frac{J}{2} \right\rfloor +1 \leq i \leq J$,
we have $0\le J-i\le \left\lfloor \frac{J}{2} \right\rfloor$ and thus
\[
a_{J-i}([\gamma +k]) = \varepsilon \cdot
a_{i}([\gamma -k-J]).
\]
Replacing $k$ with $-k-J$ in the above equality yields that equality \eqref{eq:gamma-symmetry} also holds for $\left\lfloor \frac{J}{2} \right\rfloor +1 \leq i \leq J$.
\end{rem}

By the symmetry property \eqref{eq:gamma-symmetry} in the definition of $q$-power-partibility, it is direct to obtain the following proposition.
\begin{prop}\label{prop:L}
If $L$ is $q$-power-partible, we have $\hdeg L=-\ldeg L$.
\end{prop}

Next, we will show that when $L$ is $q$-power-partible, one can find Laurrent polynomials $u([k])$ such that $L^{\ast}(u([k]))$ satisfy certain symmetry property.
\begin{theo}\label{v=cv}
	Suppose that $L=\sum_{i=0}^{J} a_i([k])\sigma^i\in\bF[q^k,q^{-k}][\sigma]$ is $q$-power-partible with respect to $(\gamma,\varepsilon)$
and for some $\tilde{\varepsilon} \in \{1,-1\}$, $u([k])\in\bF[q^k,q^{-k}]$ satisfies
\begin{equation}\label{eq:u}
u([\gamma+k]) = \tilde{\varepsilon}\cdot u([\gamma-k-J]).
\end{equation}
Then $v([k])=L^{\ast}(u([k]))$ satisfies
\begin{equation*}\label{eq:v}
v([\gamma+k])=c \cdot  v([\gamma-k]),
\end{equation*}
where $c=\varepsilon\tilde{\varepsilon}\in \{1,-1\}$.
\end{theo}
\pf
Since $L$ is $q$-power-partible with respect to $(\gamma,\varepsilon)$, we have
\begin{equation}\label{eq:qpp}
a_i([\gamma + k-i]) = \varepsilon \cdot a_{J-i}([\gamma - k +i- J]), \quad 0\leq i\leq J.
\end{equation}
Clearly  $v([k])=L^{\ast}(u([k])) = \sum_{i=0}^J a_{i}([k-i]) u([k-i])$.
Then equalities \eqref{eq:u} and \eqref{eq:qpp} lead to
\begin{align*}
	v([\gamma+k])
	&= \sum_{i=0}^{J} a_i([\gamma+k-i]) u([\gamma+k-i]) \\
	&= \sum_{i=0}^{J} \varepsilon\tilde{\varepsilon} \cdot a_{J-i}([\gamma-k+i-J]) u([\gamma-k+i-J]) \\
	&= c  \sum_{i=0}^{J} a_{J-i}([\gamma-k-(J-i)]) u([\gamma-k-(J-i)]) \\
	&= c  \sum_{i=0}^{J} a_i([\gamma-k-i]) u([\gamma-k-i]) \\
	&= c \cdot v([\gamma-k]).
\end{align*}
\qed

Given $q^\gamma,q^\delta\in\bF$ and $c\in\{1,-1\}$,
denote
\begin{equation*}\label{Space}
S_{\gamma,\delta}^{\langle c \rangle}=
\{Q([k])\in\bF[q^k,q^{-k}] \mid Q([\gamma+k])=c\cdot Q([\delta-k])\}.
\end{equation*}
The set $S_{\gamma,\delta}^{\langle c \rangle}$ is clearly a linear space over $\bF$ and will be written as $S_{\gamma}^{\langle c \rangle}$ for short when $\gamma=\delta$.

An explicit characterization of Laurent polynomials in $S_{\gamma,\delta}^{\langle c \rangle}$ is given as follows.
\begin{prop}\label{prop:T-symmetry}
For a Laurent polynomial $Q([k])\in\bF[q^k,q^{-k}]$,
$Q([k])\in S_{\gamma,\delta}^{\langle c \rangle}$ if and only if it can be written as
\begin{equation*}\label{eq:coe}
Q([k]) = \sum_{i=-I}^{I} e_i q^{ki}
\end{equation*}
for some $I\in\bN$, $e_i\in\bF$ and
$e_i=c\cdot q^{-(\gamma +\delta)i}  e_{-i}$, $0\leq i\leq I$.
\end{prop}
\pf
As $Q([k])$ is a Laurent polynomial in $q^k$, one can write it in the form
\[
Q([k]) = \sum_{i=-I_1}^{I_2} e_i q^{ki}, \quad
I_1,I_2 \in \mathbb{N}, e_i \in \bF,
\]
with $e_{-I_1} e_{I_2} \neq 0$.
Then it can be seen that
\[Q([\gamma+k]) = \sum_{i=-I_1}^{I_2} e_i q^{\gamma i} q^{ki}\]
and
\[
Q([\delta-k]) = \sum_{i=-I_1}^{I_2} e_i q^{\delta i} q^{-ki}
	= \sum_{i=-I_2}^{I_1} e_{-i} q^{-\delta i} q^{ki}.
\]
Then $Q([\gamma+k]) = c \cdot Q([\delta-k])$  forces $I_1 = I_2 =I$ for some $I\in\bN$ and
\[
e_iq^{\gamma i} = c  \cdot e_{-i}q^{-\delta i}, \quad -I \le i \le I,
\]
and vice versa.
\qed

\begin{rem}
If $c=-1$, then $e_0(q) = -e_0(q)$ leads to $e_0(q)=0$.
\end{rem}

Now we are able to illustrate the process of $q$-power-partible reduction.
\begin{theo}\label{eq:qpp-process}
Suppose that $Q([k])\in S_{\gamma}^{\langle c \rangle}$ and $L\in\bF[q^k,q^{-k}][\sigma]$ is nondegenerated with order $J$ and $q$-power-partible with respect to $(\gamma,\varepsilon)$.
Then there exist $p([k])\in S_{\gamma,\gamma-J}^{\langle c\varepsilon\rangle}$ and $\tilde{Q}([k])\in S_{\gamma}^{\langle c \rangle}$ with $\deg \tilde{Q}(x)\leq \deg L$ such that
\begin{equation}\label{eq:qpp-reduction}
Q([k])=L^{\ast}(p([k]))+\tilde{Q}([k]).
\end{equation}
When $c\varepsilon=1$, we can further guarantee that $\deg \tilde{Q}(x)<\deg L$.
\end{theo}
\pf
Suppose now $\deg Q(x)>\deg L$.
By Proposition \ref{prop:L} and Proposition \ref{prop:T-symmetry}, we can assume that
\[
\hdeg L=-\ldeg L=d
\quad \text{and}\quad
\hdeg Q(x)=-\ldeg Q(x)=I
\]
for some $d,I\in\bN$ with $I>d$.
Let
\[
p_0([k])=\begin{cases}
           1, & \mbox{if } c\varepsilon=1 \\
           0, & \mbox{otherwise}.
         \end{cases}
\]
and
\begin{equation*}
p_i([k])=q^{-ki}+c\varepsilon\cdot q^{(J-2\gamma)i}q^{ki},\ i\in\bN^{+}.
\end{equation*}
Proposition \ref{prop:T-symmetry} shows that $p_i([\gamma+k])=c\varepsilon \cdot p_i([\gamma-k-J])$, that is $p_i([k])\in S_{\gamma,\gamma-J}^{\langle c\varepsilon\rangle}$ for $i\in\bN$.
Thus $v_i([k])=L^{\ast}(p_i([k]))\in S_{\gamma}^{\langle c \rangle}$ by Theorem \ref{v=cv} and
\[
\hdeg v_i(x)=d+i,\quad \ldeg v_i(x)=-d-i
\]
by Theorem \ref{thm:degree}.
Then there exists $b_{I-d}\in\bF$ so that
\[
Q_{0}([k])=Q([k])-b_{I-d}v_{I-d}([k])\in S_{\gamma}^{\langle c \rangle}
\]
and
$I_0=\hdeg Q_{0}(x)=-\ldeg Q_{0}(x)<I$.

We first consider the case when $c\varepsilon=-1$.
If $I_0\leq d$, we are done with the decomposition
$Q([k])=L^{\ast}(b_{I-d}p_{I-d}([k]))+Q_{0}([k])$.
Otherwise, we continue the above process to $Q_{0}([k])$ to find $b_{I_0-d}\in\bF$ so that
\[
Q_{1}([k])=Q_0([k])-b_{I_0-d}v_{I_0-d}([k])\in S_{\gamma}^{\langle c \rangle}
\]
and
$I_1=\hdeg Q_{1}(x)=-\ldeg Q_{1}(x)<I_0$.
If $I_1\leq d$, we are done with
$Q([k])=L^{\ast}(p([k]))+Q_{1}([k])$, where
$p([k])=b_{I-d}p_{I-d}([k])+b_{I_0-d}p_{I_0-d}([k])$.
Otherwise, we continue the above process to $Q_{1}([k])$.
After finite steps, we find a decomposition satisfying
\eqref{eq:qpp-reduction}.

In the case of $c\varepsilon=1$, we can reduce one step further by the Laurent polynomial $v_0([k])=L^{\ast}(p_0([k]))$ of degree $d$ to get a $\tilde{Q}([k])$ with $\deg \tilde{Q}(x)<\deg L$.
This completes the proof.
\qed

\subsection{Applications}
Recall that Pan's $q$-congruence \eqref{eq:Pan-A_k} implies that
\begin{equation}\label{eq:Pan2014D1}
	\sum_{k=0}^{n-1} q^{n-1-k} [2k+1] D_{k}(q) \equiv 0 \pmod{[n]},
\end{equation}
where $D_k(q)$ are the $q$-central Delannoy numbers given by \eqref{q-Delannoy}.
It is straightforward to check that the annihilator $L$ of $D_k(q)$ given by \eqref{qDe1L} with $\varepsilon=1$ is not HDD and $q$-power-partible with respect to $(-\frac{1}{2},-1)$.

For any $\gamma,b \in \bQ$ and $m \in \bN$, we define the Laurent polynomial
\begin{equation}\label{f_{m,b}}
  f_{m,b}^{(\gamma)}([k]) = q^{m(-k+b)}[2k-2\gamma]^m = ([k+b-2\gamma] - [-k+b])^m.
\end{equation}
Note that
$f_{m,b}^{(\gamma)}([k])\in  S_{\gamma}^{\langle(-1)^m\rangle}$ and $q^{(n-1-k)(2v+1)} [2k+1]^{2v+1}=f_{2v+1,n-1}^{(-\frac{1}{2})}([k])$.
This motivates the discovery of the following generalization of \eqref{eq:Pan2014D1}.
\begin{theo}\label{Pan-generalization}
	For any $v\in \bN$, we have
	\[
	\sum_{k=0}^{n-1} q^{(n-1-k)(2v+1)} [2k+1]^{2v+1} D_k(q) \equiv 0 \pmod{[n]}.
	\]
\end{theo}

For the proof of Theorem \ref{Pan-generalization}, we need the following two lemmas.
The first one gives an explicit expansion of $f_{m,b}^{(\gamma)}([k])$.
\begin{lem}\label{eq:fmb-coe}
Let $f_{m,b}^{(\gamma)}([k])$ be defined as in \eqref{f_{m,b}}.
Then
\begin{equation*}
f_{m,b}^{(\gamma)}([k])=\frac{q^{mb}}{(1-q)^{m}}\sum_{\substack{i=-m\\2\mid (m+i)}}^{m} \binom{m}{\frac{m+i}{2}} (-1)^{\frac{m+i}{2}} q^{-\gamma (m+i)} \ q^{ki}.
\end{equation*}
\end{lem}
\pf
By definition, one can see
\begin{align*}
	f_{m,b}^{(\gamma)}([k])
	&= \frac{q^{m(-k+b)} }{(1-q)^{m}} (1-q^{2k-2\gamma})^{m} \\
	&= \frac{q^{m(-k+b)}}{(1-q)^{m}}\sum_{j=0}^{m} \binom{m}{j} (-1)^{j} q^{2j(k-\gamma)} \\
	&= \frac{q^{mb}}{(1-q)^{m}}\sum_{j=0}^{m} \binom{m}{j} (-1)^{j} q^{-2\gamma j} \, q^{k(2j-m)}\\
    &= \frac{q^{mb}}{(1-q)^{m}}\sum_{\substack{i=-m\\2\mid (m+i)}}^{m} \binom{m}{\frac{m+i}{2}} (-1)^{\frac{m+i}{2}} q^{-\gamma (m+i)} \ q^{ki}
\end{align*}
by taking $i=2j-m$ in the final step.
\qed

\begin{lem}
Let $v,n\in\bN^{+}$ and $L$ be given by \eqref{qDe1L} with $\varepsilon=1$.
Then there exist $c(q),c_i(q)\in\bZ[q,q^{-1}]$ such that
\begin{equation}\label{eq:L^ast}
f_{2v+1,n-1}^{(-\frac{1}{2})}([k])=\sum_{i=0}^{2v-1} \frac{c_i(q)q^{(2v+1)n}}{(1-q)^{2v}}  L^{\ast}(p_i([k]))
+\frac{c(q)q^{2vn}}{(1-q)^{2v}}f_{1,n-1}^{(-\frac{1}{2})}([k]),
\end{equation}
where $p_i([k])=q^{-ki}+q^{(k+3)i}$ for $i\in\bN^{+}$ and $p_0([k])=1$.
\end{lem}
\pf
Note that
$f_{2v+1,n-1}^{(-\frac{1}{2})}([k])=q^{(n-1-k)(2v+1)} [2k+1]^{2v+1}\in  S_{-\frac{1}{2}}^{\langle -1 \rangle}$ and $\hdeg{L}=-\ldeg{L}=2$.
Next, we go through the $q$-power partible reduction given in
the proof of Theorem \ref{eq:qpp-process}.
By the definition of $p_i([k])$, one can see $\tilde{p}_i([k])=L^{\ast}(p_i([k]))\in S_{-\frac{1}{2}}^{\langle -1 \rangle}$ for all $i\in\bN$.
Moreover, since
\begin{align*}
  \tilde{p}_i([k])=&q^{-k}(1+q^{k+2})[k+1]p_i([k])+q^{-k+2}(1+q^{k-1})[k]p_i([k-2])\\
           &+(q+q^{k+1}+q^{2k+2}+3q^{k+2})[-2k-
          1]p_i([k-1]).
\end{align*}
We have $(1-q)\tilde{p}_i([k])\in\bZ[q,q^{-1}][q^{k},q^{-k}]$,
$\hdeg \tilde{p}_i(x)=-\ldeg \tilde{p}_i(x)=i+2$ and
\begin{equation*}\label{eq:v_i-degree}
[q^{-k(i+2)}]\left((1-q)\tilde{p}_i([k])\right)=-q^{i}.
\end{equation*}
Lemma \ref{eq:fmb-coe} shows that $q^{-(2v+1)n} (1-q)^{2v+1}f_{2v+1,n-1}^{(-\frac{1}{2})}([k])\in \bZ[q,q^{-1}][q^{k},q^{-k}]$.
So there exists $c_{2v-1}(q)\in \bZ[q,q^{-1}]$ such that
\[
Q_0([k])=q^{-(2v+1)n}(1-q)^{2v+1}f_{2v+1,n-1}^{(-\frac{1}{2})}([k])
-c_{2v-1}(q)(1-q)\tilde{p}_{2v-1}([k])
\]
satisfies
\[Q_0([k])\in  S_{-\frac{1}{2}}^{\langle -1 \rangle}\cap \bZ[q,q^{-1}][q^{k},q^{-k}]\]
and
\[\hdeg{Q_0}(x)=-\ldeg{Q_0(x)}\le 2v.\]
Then there exists $c_{2v-2}(q)\in\bZ[q,q^{-1}]$ such that
\[
Q_1([k])=Q_0([k])-c_{2v-2}(q)(1-q)\tilde{p}_{2v-2}([k])\in  S_{-\frac{1}{2}}^{\langle -1 \rangle}\cap \bZ[q,q^{-1}][q^{k},q^{-k}]
\]
with $\hdeg{Q_1}(x)=-\ldeg{Q_1}(x)\le 2v-1$.
Continuing the above process, we find
\begin{equation}\label{eq:L^ast1}
q^{-(2v+1)n}(1-q)^{2v+1}f_{2v+1,n-1}^{(-\frac{1}{2})}([k])
=\sum_{i=0}^{2v-1}c_i(q)(1-q)L^{\ast}(p_i([k]))
+\tilde{Q}([k])
\end{equation}
for some $c_i(q)\in\bZ[q,q^{-1}]$
and $\tilde{Q}([k])\in S_{-\frac{1}{2}}^{\langle -1 \rangle}\cap\bZ[q,q^{-1}][q^{-k},q^k]$
satisfying $\hdeg\tilde{Q}(x)=-\ldeg\tilde{Q}(x)<2$.
By Proposition \ref{prop:T-symmetry},
\begin{equation}\label{eq:Q}
\tilde{Q}([k])=\tilde{c}(q)(q^{-k}-q^{k+1})=\tilde{c}(q)(1-q)q^{1-n}
f_{1,n-1}^{(-\frac{1}{2})}([k])
\end{equation}
for some $c(q)\in \bZ[q,q^{-1}]$.
Substituting \eqref{eq:Q} into \eqref{eq:L^ast1} and letting $c(q)=\tilde{c}(q)q$ completes the proof.
\qed

\noindent\emph{Proof of Theorem \ref{Pan-generalization}.}
Multiplying both sides of \eqref{eq:L^ast} with $D_k(q)$ and then summing over $k$ from $0$ to $n-1$.
The conclusion comes directly by Corollary \ref{LM:q-Delannoy} and \eqref{eq:Pan2014D1}.
\qed

\noindent \textbf{Acknowledgments.}
This work was supported by the National Natural Science Foundation of China (No. 12101449, 12271511 and 12271403).
All authors contribute equally and the names of the authors are listed in alphabetical order according to their family names.

\end{document}